\documentclass[10pt]{article}
\usepackage{amsmath,amssymb}
\usepackage{amscd}
\usepackage{amsthm}
\usepackage{stmaryrd}
\usepackage[dvipdfm]{graphicx}
\usepackage{makeidx}
\theoremstyle{theorem}
\newtheorem{defn}{Definition}[section]
\newtheorem{thm}[defn]{Theorem}
\newtheorem{prop}[defn]{Proposition}
\newtheorem{lem}[defn]{Lemma}

\newtheorem{fac}[defn]{Fact}

\DeclareMathOperator{\tr}{tr}
\DeclareMathOperator{\Ker}{Ker}

\DeclareMathOperator{\Ric}{Ric}

\DeclareMathOperator{\diam}{diam}
\DeclareMathOperator{\Vol}{Vol}

\makeatletter 
\def\osc{\mathop{\operator@font osc}}
\makeatother

\title{On the uniqueness of Sasaki-Einstein metrics}

\author{Ken'ichi Sekiya
\footnote{
Department of Mathmatics, Graduate School of Science, Osaka University\qquad\qquad\break
}
}
\date{}

\makeindex
\begin{document}

\maketitle

\begin{abstract}
Let $S$ be a compact Sasakian manifold which
does not admit non-trivial Hamiltonian holomorphic vector fields.
If there exists an Einstein-Sasakian metric on $S$, then it is unique. 
\end{abstract}

\section{Introduction}
The aim of this paper is to show the uniquness theorem of positive Sasaki-Einstein metrics. 
An Sasaki-Einstein manifold admits a one dimensional Reeb foliation
with a transversal K\"ahler-Einstein metric,
which is studied from many view points 
between geometry and mathematical physics.
Boyer, Galichi and Koll\'ar obtained Sasaki-Einstein metrics on a family of the links of hypersurfaces of Brieskorn-Pham type, which includes exisotic spheres. 
Guantlett, Martelli, Sparks and Waldram discoverd that there exist irregular toric Einstein-Sasaki metrics which are not obtained as 
total spaces of orbibundles on Einstein-K\"ahler orbifolds (\cite{GMSW1}, \cite{GMSW2}).
These toric examples are much explored and 
Futaki, Ono and Wang showed that every toric positive Sasakian manifold admits 
Sasaki-Einstein metrics(\cite{F1}). 
On a compact K\"ahler manifold with positive first Chern class, Bando and Mabuchi proved the uniqueness 
theorem of K\"ahler-Einstein metrics (\cite{BM1}). 
K. Cho, A. Futaki and H. Ono proved that the toric Einstein-Sasaki metric is unique up to the automorphism of a toric Sasakian manifold (\cite{CFO1}).
In the present paper, we show the following theorem,
\begin{thm}
Let $(S,\xi,\eta,\Phi)$ be a compact Sasakian manifold. 
We assume that $S$ doesn't admit nontrivial Hamiltonian holomorphic vector fields.
If $S$ has a Sasaki-Einstein metric, then the Sasaki-Einstein metric is unique. 
In other words,
if there are two Sasaki-Einstein metrics $\omega_{1}$ and $\omega_{2}$ on $S$, then $\omega_{1}=\omega_{2}$.
\end{thm} 
Our method is a generalization of Bando-Mabuchi's argument to Sasakian geometry.
We construct functionals $L$, $M$, $I$ and $J$ on the space of Sasakian structures with  
basic first Chern class.
These functionals satisfy the suitable properties as in K\"ahler geometry.
The problem of Sasaki-Einstein metrics reduces to solving the following Monge-Amp\`ere equation which gives rise to 
transversal K\"ahler-Einstein metrics with positive Ricci curvature, 
\begin{equation*}
\frac{(d\eta+\sqrt{-1}\partial_{B}\bar{\partial}_{B}u)^{m}}{(d\eta)^{m}}=\exp(-(2m+2)u+h)
\end{equation*}

The key point is to show the a priori estimate of $C^0$-norm of solutions $u$ of the Monge-Amp\`ere equation
and an intriguing point is an estimate of infimum of $u$ (lemma \ref{osc-a}).
The Monge-Amp\`ere equation only gives the transversal Ricci curvature which does not lead a lower bound of the Ricci curvature 
by a positive constant.
There is a difficulty of the $C^0$-estimate of $u_t$ since we cannot apply the Myers theorem directly to obtain an estimate of the daimeter of $S$.
We introduce a family of Sasakian structures $\{ g_{u,\mu}\}$ whose contact forms are
given by the multiplication of positive constant 
$\mu^{-1}$.
Under a suitable choice of $\mu$, it follows that the Ricci curvature of $g_{\mu,\lambda}$ is bounded from below by a positive constant. Thus
we can control their diameters by the Myeres theorem. 
An estimate of their volumes together with their diameters
 gives rise to 
 the desired estimate of solutions $u$ by using the estimate of the Green functions
(see lemma \ref{osc-a} for more detail).
Our method of the estimate is simple and effective in transversal K\"ahler metrics, which slightly different from 
the ordinary argument in K\"ahler geometry
as transversal K\"ahler classes of the family $\{ g_{u,\mu}\}$ are changing.

It must noted that Nitta obtained the theorem of uniqueness of Einstein -Sasakian metrics independently by the different method 
which heavily depends on several 
results in sub-Riemannian geometry such as the regularity of the space of piece-wise smooth horizontal paths (\cite{Ni}). 
 
An advantage of our method is that it is self-contained and 
 could be generalized to 
more general transversal K\"ahler geometry which includes 
$3$-Sasakian manifolds.

I would like to thank Ryushi Goto for helpful advice. 

\section{Sasakian manifold}

In this section we give a brief explanation of Sasakian manifolds

\begin{defn}
Let $(S,g)$ be a Riemannian manifold of dimension $2m+1$ and $C(S)$ the cone $S\times \mathbb{R}_{>0}$ with $r\in\mathbb{R}_{>0} $.
A Riemannian manifold $(S,g)$ is said to be a Sasakian manifold if the cone manifold $(C(S),\bar{g})=(S\times \mathbb{R}_{>0},dr^{2}+r^{2}g)$ is a K\"ahler manifolds with complex structure $J$ which satisfies
\begin{equation*}
\mathcal{L}_{r\frac{\partial}{\partial r}}J=0,
\end{equation*}
where $\mathcal{L}_{r\frac{\partial}{\partial r}}J$ denotes the Lie derivative of $J$ by the vector field $r\frac{\partial}{\partial r}$.
\end{defn}

A Sasakian manifold $S$ is often identified with the submanifold $\{r=1\}=S\times \{1\} \subset C(S)$.
Note that $C(S)$ is a real $2m$ dimensional manifold.

\begin{defn}
We define a vector field $\xi$ on $S$ and a $1$-form $\eta$ on $S$ by
\begin{equation*}
\xi=J\left(r\frac{\partial}{\partial r}\right),\qquad \eta(Y)=g(\xi,Y)
\end{equation*}
where $Y$ is a smooth vector field on $S$.
The vector field $\xi$ is the {\it Reeb field}.
We denote by $\mathcal{F}_{\xi}$
the $1$-dimensional foliation generated by $\xi$ which is called the {\it Reeb foliation}.
\end{defn}
Then we see that 
\begin{equation}\label{eq: xi, eta, contact}
\eta(\xi)=1,\qquad i_{\xi}d\eta=0,\qquad (d\eta)^{m}\wedge\eta\neq 0.
\end{equation}
The 1-form $\eta$ is a contact form on $S$ which defines a $2m$-dimensional subbundle $D$ of the tangent bundle $TS$, 
where at each point $p \in S$ the fiber $D_{p}$ of $D$ is given by
\begin{equation*}
D_{p}=\Ker \eta_{p}.
\end{equation*}
We call $D$ the contact bundle.
The contact bundle $D$ gives the orthogonal decomposition of the tangent bundle $TS$
\begin{equation*}
TS=D\oplus L_{\xi}
\end{equation*}
where $L_{\xi}$ is the trivial bundle generated by the Reeb field $\xi$.
A Sasakian manifold $S$ is a foliated manifold with transversally K\"ahler structure.
Then $S$ admits foliated coordinates $\{ U_\alpha\}$ compatible to the structure. 
The system of coordinates consists of an open covering 
$\{U_{\alpha}\}$ of $S$ and a submersion
$\pi_{\alpha}\colon U_{\alpha}\rightarrow V_{\alpha}\subset \mathbb{C}^{m}$  for each $\alpha$
such that  
\begin{equation*}
\pi_{\alpha}\circ \pi_{\beta}=\pi_{\beta}(U_{\alpha}\cap U_{\beta})\rightarrow \pi_{\alpha}(U_{\alpha}\cap U_{\beta}),\qquad 
U_{\alpha}\cap U_{\beta}\neq 0
\end{equation*}
is biholomorphic, where $V_\alpha$ is an open set of $\mathbb C^m$.
On each $V_{\alpha}$ there is a K\"ahler structures given by the following.
The restriction of the Sasaki metric $g$ to $D$ gives a well-defined Hermitian metric $g_{\alpha}^{T}$ on $V_{\alpha}$
under the canonical isomorphism 
\begin{equation*}
d\pi_{\alpha} \colon D_{p}\rightarrow T_{\pi_{\alpha}(p)}V_{\alpha}
\end{equation*}
for any $p \in U_{\alpha}$.
Hence we have the transversally Hermitian structure on $S$.
Let $(z^{1},z^{2},\ldots,z^{m})$ be the local holomorphic coordinates on $V_{\alpha}$.
We pull back these to $U_{\alpha}$ and still write them as $(z^{1},z^{2},\ldots,z^{m})$.
Let $x$ be the coordinate along the leaves with $\xi =\frac{\partial}{\partial x}$
Then $(x,z^{1},z^{2},\ldots,z^{m})$ form local coordinates on $U_{\alpha}$.
We denote by $(D\otimes\mathbb{C})^{p,q}$ the set of forms 
of type $(p,q)$ on $S$. Then
$(D \otimes \mathbb{C})^{1,0}$ is spanned by the vectors of the form
\begin{equation*}
\frac{\partial}{\partial z^{i}}-\eta\left(\frac{\partial}{\partial z^{i}}\right)\xi,\qquad i=1,2,\ldots,m.
\end{equation*}
Since $i_{\xi}d\eta=0$,
\begin{equation*}
d\eta\left(\frac{\partial}{\partial z^{i}}-\eta\left(\frac{\partial}{\partial z^{i}}\right)\xi,
\overline{\frac{\partial}{\partial z^{j}}-\eta\left(\frac{\partial}{\partial z^{j}}\right)\xi}\right)
=d\eta\left(\frac{\partial}{\partial z^{i}},\overline{\frac{\partial}{\partial z^{j}}}\right).
\end{equation*}
Thus the fundamental 2-form $\omega_{\alpha}$ of the Hermitian metric $g_{\alpha}^{T}$ on $V_{\alpha}$
is the same as the restriction of $d\eta$ to the slice $\{x={\rm constant}\}$ in $U_{\alpha}$.
Since the restriction of a closed 2-form to a submanifold is closed, then $\omega_{\alpha}$ is closed.
By this construction
\begin{equation*}
\pi_{\alpha}\circ\pi_{\beta}^{-1}\colon \pi_{\beta}(U_{\alpha}\cap U_{\beta})\rightarrow \pi_{\alpha}(U_{\alpha}\cap U_{\beta})
\end{equation*}
gives an isometry of the K\"ahler structure.

\begin{defn}\label{def: foliation chart}
The collection of K\"ahler metrics $\{g_{\alpha}^{T}\}$ on $\{V_{\alpha}\}$ is called a transverse K\"ahler metric.
Since they are isometric over the intersections we suppress $\alpha$ and denote it by $g^{T}$. We call coordinates system $(x,z^{1},z^{2},\ldots,z^{m})$ given above a foliation chart.
\end{defn}

We also write $\Ric^{T}$ and $s^{T}$ for Ricci curvature of $g^{T}$ and scalar curvature of that.
It should be emphasized that, though $g^{T}$ are defined only locally on each $V_{\alpha}$,
the pull-back to $U_{\alpha}$ of the K\"ahler forms $\omega_{\alpha}$ on $V_{\alpha}$ patch together
and coincide with the global form $d\eta$ on $S$,
and $d\eta$ can even be lifted to the cone $C(S)$ by pull-back.
For this reason we often refer to $d\eta$ as the K\"ahler form of the transverse K\"ahler form of the transverse K\"ahler metric $g^{T}$.
The next is a well known result.

\begin{thm}[\cite{F1}]\label{R-R^{T}}
Let $(S,g)$ be a Sasakian manifold.
Then, we have
\begin{align*}
\Ric(X,\,\xi\,)&=2m\,\eta(X) ,\qquad \forall X \in TS\\
\Ric(X,Y)&=\Ric^{T}(X,Y)-2g(X,Y),\qquad\forall X,Y\in D
\end{align*}
\end{thm}

\begin{defn}
A Sasakian manifold $(S,g)$ is $\eta$-Einstein if there are two constants $\lambda$ and $\nu$ such that
\begin{equation*}
\Ric=\lambda g+\nu\,\eta\otimes\eta.
\end{equation*}
\end{defn}

\begin{defn}
A Sasaki-Einstein manifold is a Sasakian manifold $(S,g)$ with $\Ric=2mg$.
\end{defn}

\begin{defn}
A Sasakian manifold $S$ is said to be transversely K\"ahler-Einstein Sasaki manifold if
\begin{equation*}
\Ric^{T}=\tau g^{T}
\end{equation*}
for some real constan $\tau$.
\end{defn}

It is well-known that if $S$ is a transversely K\"ahler-Einstein Sasaki manifold if and only if $(S,g)$ is $\eta$-Einstein (cf\cite{BGM1}).
In fact, if $\Ric^{T}=\tau g^{T}$ then
\begin{equation*}
\Ric =(\tau-2)g+(2m+2-\tau)\eta\otimes \eta.
\end{equation*}
Conversely if $\Ric=\lambda g +\nu\eta\otimes\eta$ then
\begin{equation*}
\Ric^{T}=(\lambda +2)g^{T}.
\end{equation*}

\section{Basic form}

We introduce basic forms on Sasakian manifolds which is relevant to transversely K\"ahler-Einstein metrics on them.
Let $S$ be a compact Sasakian manifold of dimension $2m+1$.

\begin{defn}\label{basic forms}
A $p$-form $\alpha$ on $S$ is said to be basic if the following conditions hold 
\begin{equation*}
i_{\xi}\alpha=0,\qquad \mathcal{L}_{\xi}\alpha=0.
\end{equation*}
Let $\Lambda^{p}_{B}$ be the sheaf of germs of basic $p$-forms and $\Omega^{p}_{B}$ the set of all global sections of $\Lambda^{p}_{B}$.
\end{defn}

It follows from (\ref{eq: xi, eta, contact}) that $d\eta$ is a basic form.
Let $(x,z^{1},\ldots,z^{m})$ be the foliation chart on $U_{\alpha}$ as in definition \ref{def: foliation chart}.
Then we write
\begin{equation*}
\sum \alpha_{i_{1}\ldots i_{p}\bar{j}_{1}\ldots \bar{j}_{q}}
dz^{i_{1}}\wedge\cdots\wedge dz^{i_{p}}\wedge d\bar{z}^{j_{1}}\wedge\cdots\wedge d\bar{z}^{j_{q}}
\end{equation*}
for a form of type $(p,q)$ on $U_{\alpha}$.
If $U_{\alpha}\cap U_{\beta}\neq \emptyset$ and $(y,w^{1},\ldots,w^{m})$ is the foliation chart on $U_{\beta}$, then
\begin{equation*}
\frac{\partial z^{i}}{\partial \bar{w}^{j}}=0,\qquad \frac{\partial z^{i}}{\partial y}=0.
\end{equation*}
Therefore, as in K\"ahler geometry, we have a notion of forms of type $(p,q)$ which is independent of a choice of charts.
If $\alpha$ is basic, then the coefficient $\alpha_{i_{1}\ldots i_{p}\bar{j}_{1}\ldots \bar{j}_{q}}$ is a function which dose not depend on $x$.
Thus we have a well-defined operators
\begin{align*}
\partial_{B}\colon& \Lambda_{B}^{p,q}\rightarrow \Lambda_{B}^{p+1,q}\\
\bar{\partial}_{B}\colon& \Lambda_{B}^{p,q}\rightarrow \Lambda_{B}^{p,q+1}.
\end{align*}
It follows that $d\alpha$ is basic for a basic form $\alpha$.
Hence the exterior derivative $d$ preserves the basic forms and we have the basic exterior derivative $d_B$ and the complex of basic forms, 
\begin{equation*}
\cdots\to \Omega_B^p\to \Omega_B^{p+1}\to \cdots
\end{equation*}
which gives the basic cohomology group $H^p_B(S)$.
We denote by $[\alpha]_B$ the basic cohomology class  represented by a $d_B$-closed, basic $p$-form .

As in K\"ahler geometry, we have the decomposition $d_{B}=\partial_{B}+\bar{\partial}_{B}$.
Let $d_{B}^{c}=\frac{\sqrt{-1}}{2}(\bar{\partial}_{B}-\partial_{B})$.
It is clear that
\begin{equation*}
d_{B}d_{B}^{c}=\sqrt{-1}\partial_{B}\bar{\partial}_{B},\qquad d_{B}^{2}=(d_{B}^{c})^{2}=0
\end{equation*}
Let $\partial_{B}^{*}$ be the adjoint operator of $\partial_{B}$ and $\bar{\partial}_{B}^{*}$ the adjoint operator of $\bar{\partial}_{B}$ with respect to the transversally K\"ahler metric $g^T$.
The basic Laplacian and the basic Dolbeault Laplacian are defined by
\begin{align*}
\triangle^{B}&=d_{B}^{*}d_{B}+d_{B}d_{B}^{*}\\
\square^{B}&=\bar{\partial}_{B}^{*}\bar{\partial}_{B}+\bar{\partial}_{B}\bar{\partial}_{B}^{*}.
\end{align*}
On a Sasakian maifold, the $\partial\overline{\partial}$-lemma holds for basic forms.

\begin{prop}[\cite{E1}]\label{p-lem}
Let $\alpha$ and $\beta$ be
two basic forms of type $(1,1)$ on a compact Sasakian manifold $S$ with $[\alpha]_{B}=[\beta]_{B}\in H^2_B(S)$
Then there is a basic function $h$ such that
\begin{equation*}
\alpha=\beta+\sqrt{-1}\partial_{B}\bar{\partial}_{B}h.
\end{equation*}
\end{prop}

As in \cite{F1} a new Sasakian structure fixing $\xi$ and varying $\eta$ is given by
\begin{equation*}
\eta_{\varphi}=\eta+d_{B}^{c}\varphi
\end{equation*}
where $\varphi$ is a small basic function preserving the positivity 
condition.
Since $\dim_{\mathbb{R}}S=2m+1$, a basic form with degree more than $2m+1$ is zero.
Since $d\eta$ is basic, it follows from the Stokes theorem that we have
\begin{equation*}
\int_{S} d_{B}\alpha\wedge\beta\wedge\eta=-(-1)^{\deg\alpha}\int_{S}\alpha\wedge d_{B}\beta\wedge\eta,
\end{equation*}
where $\alpha$, $\beta$ are basic and $\deg\alpha+\deg\beta=2m-1.$
Since $d^c_B\varphi$ is basic, for a basic $2m$-form $\gamma$, we also have
\begin{equation*}
\int_{S}\gamma\wedge\eta=\int_{S}\gamma\wedge\eta_{\varphi}.
\end{equation*}
Therefore, in virtually, the results in K\"ahler geometry which can be proved only using the Stokes theorem, including the integration by parts still holds on compact Sasaki manifolds by using the contact form $\eta$. 

\begin{lem}\label{volume}
\begin{equation*}
\int_{S}(d\eta)^{m}\wedge\eta=\int_{S}(d\eta_{\varphi})^{m}\wedge\eta_{\varphi}
\end{equation*}
\end{lem}
\begin{proof}
\begin{align}
\int_{S}(d\eta_{\varphi})^{m}\wedge\eta_{\varphi}
&=\int_{S}(d\eta+dd_{B}^{c}\varphi)^{m}\wedge(\eta+d_{B}^{c}\varphi)\\
&=\int_{S}\sum\left(\begin{array}{c}m\\k\end{array}\right)(d\eta)^{m-k}\wedge(d_{B}d_{B}^{c}\varphi)^{k}\wedge(\eta+d_{B}^{c}\varphi)\\
&=\int_{S}\sum\left(\begin{array}{c}m\\k\end{array}\right)(d\eta)^{m-k}\wedge(d_{B}d_{B}^{c}\varphi)^{k}\wedge\eta\label{eq: 3-4}\\
&\qquad+\int_{S}\sum\left(\begin{array}{c}m\\k\end{array}\right)(d\eta)^{m-k}\wedge(d_{B}d_{B}^{c}\varphi)^{k}\wedge d_{B}^{c}\varphi. \label{eq: 3-5}
\end{align}
The last term (\ref{eq: 3-5})  is zero because it is a basic $(2m+1)$-form.
In the case of $k\geq 1$, the term (\ref{eq: 3-4}) is given by 
\begin{align}
&d\left((d\eta)^{m-k}\wedge d_{B}^{c}\varphi\wedge(d_{B}d_{B}^{c}\varphi)^{k-1}\wedge\eta\right)\\
&=(d\eta)^{m-k}\wedge(d_{B}d_{B}^{c}\varphi)^{k}\wedge\eta-(d\eta)^{m-k}\wedge d_{B}^{c}\varphi\wedge(d_{B}d_{B}^{c}\varphi)^{k-1}\wedge d\eta.\label{eq: 7}
\end{align}
The second term of (\ref{eq: 7}) is zero because it is a basic $(2m+1)$-form.
Therefore the result follows from Stokes theorem
\begin{equation*}
\int_{S}(d\eta)^{m}\wedge\eta=\int_{S}(d\eta_{\varphi})^{m}\wedge\eta_{\varphi}.
\end{equation*}
\end{proof}

\begin{defn}\label{def: Ric}
A collection of $(1,1)$-forms $\rho_{\alpha}^{T}$ on $V_{\alpha}\subset \mathbb{C}^{m}$ is given by
\begin{equation*}
\rho_{\alpha}^{T}=-\sqrt{-1}\partial\bar{\partial}\log\det(g^{T}_{\alpha}).
\end{equation*}
Then the collection of the pullback of forms $\pi_{\alpha}^{*}\rho_{\alpha}^{T}$
defines a basic from of type $(1,1)$
which is called the transverse Ricci form.
We denote by $\Ric^{T}$the transverse Ricci form as well as the transverse Ricci tensor.
To emphasize transverse Ricci form with respect to $d\eta$, we often write $\Ric^{T}(d\eta)$.
 \end{defn}

There is a necessary condition for the existence of transversely K\"ahler-Einstein metrics.

\begin{prop}[\cite{F1}]
The transverse Ricci form $\Ric^{T}(d\eta)$ is represented by $\tau d\eta$ for some constant $\tau$ if and only if $c_{1}(D)$ is zero where $D=\Ker \eta$ is contact bundle.
\end{prop}

From now on we always assume that $c_{1}(D)=0$.

\section{Functionals on compact Sasakian manifolds}

In this section, we define functionals on compact Sasakian manifolds which are analogous to the ones in K\"ahler geometry.
We define $\Omega$ by
\begin{equation*}
\Omega=\{\varphi \;\vert\; \varphi \text{ is basic and $d\eta_{\varphi}=\eta+d^c_{B}\varphi$ is positive definite}\}.
\end{equation*}
Thus $\eta_{\varphi}$ gives a Sasakian structure for $\varphi\in \Omega$.

\begin{prop}\label{LM}
We assume that  $[\Ric^{T}(d\eta)]_{B}=(2m+2)[d\eta]_{B}$.
For every $(\varphi,\varphi') \in \Omega \times \Omega$, we define functionals $L$, $M$ by
\begin{align*}
L(\varphi,\varphi')&=\frac{1}{V}\int_{a}^{b}\left(\int_{S}\dot{\varphi}_{t}(d\eta_{\varphi_{t}})^{m}\wedge\eta_{\varphi_{t}}\right)dt\\
M(\varphi,\varphi')&=-\frac{1}{V}\int_{a}^{b}\left\{\int_{S}\dot{\varphi}_{t}(s^{T}(d\eta_{\varphi_{t}})-m(2m+2))
	(d\eta_{\varphi_{t}})^{m}\wedge\eta_{\varphi_{t}}\right\}dt,
\end{align*}
where  $V=\int_{S}(d\eta)^{m}\wedge\eta$ and $\{\varphi_{t}\;\vert\; a \leq t \leq b\}$ is an arbitrary piecewise smooth path in $\Omega$ such that $\varphi=\varphi_{a},\;\varphi'=\varphi_{b}$.
Then $L,M$ are independent of the choice of the path $\{\varphi_{t}\;\vert\; a \leq t \leq b\}$, therefore well-defined.
Moreover, $L,M$ satisfy the $1$-cocycle condition, and for all $C_{1},C_{2} \in \mathbb{R}$
\begin{align*}
L(\varphi,\varphi'+C_{2})&=L(\varphi,\varphi')+C_{2}\\
M(\varphi+C_{1},\varphi'+C_{2})&=M(\varphi,\varphi').
\end{align*}
\end{prop}

\begin{prop}\label{IJ}
For every $(\varphi,\varphi') \in \Omega \times \Omega$, we define functionals $I$, $J$ by
\begin{align*}
I(\varphi,\varphi')&=\frac{1}{V}\int_{S}(\varphi'-\varphi)\left((d\eta_{\varphi})^{m}-(d\eta_{\varphi'})^{m}\right)\wedge \eta\\
J(\varphi,\varphi')&=\frac{1}{V}\int_{a}^{b}
	\left(\int_{S}\dot{\varphi}_{t}\left((d\eta_{\varphi})^{m}-(d\eta_{\varphi_{t}})^{m}\right)\wedge \eta\right)dt,
\end{align*}
where $V=\int_{S}(d\eta)^{m}\wedge\eta$ and $\{\varphi_{t}\;\vert\; a \leq t \leq b\}$ is an arbitrary piecewise smooth path in $\Omega$ such that $\varphi=\varphi_{a},\;\varphi'=\varphi_{b}$.
Then the following statements hold.
\begin{enumerate}
\item $J(\varphi,\varphi')=-L(\varphi,\varphi')+\frac{1}{V}\int_{S}(\varphi'-\varphi)(d\eta_{\varphi})^{m}\wedge\eta$, and
	$J$ is independent of the choice of the path.
\item The functional $J$ doesn't satisfy the $1$-cocycle condition, but satisfy
	\begin{equation*}
	J(\varphi,\varphi')+J(\varphi',\varphi'')
	=J(\varphi,\varphi'')-\frac{1}{V}\left(\int_{S}(\varphi''-\varphi')\left((d\eta_{\varphi})^{m}-(d\eta_{\varphi'})^{m}\right)\wedge \eta\right).
	\end{equation*}
\item Let $C$ be a constant, then
	\begin{align*}
	I(\varphi,\varphi'+C)&=I(\varphi,\varphi')\\
	J(\varphi,\varphi'+C)&=J(\varphi,\varphi').
	\end{align*}
\item Let $\{\varphi_{t}\}$ be a family of basic functions, then
	\begin{equation*}
	\frac{d}{dt}\left(I(\varphi,\varphi_{t})-J(\varphi,\varphi_{t})\right)
	=\frac{1}{V}\int_{S}(\varphi_{t}-\varphi)\left(\square_{\varphi_{t}}^{B}\frac{d}{dt}\varphi_{t}\right)
	(d\eta_{\varphi_{t}})^{m}\wedge\eta.
	\end{equation*}
\item $I,I-J,J$ are non-negative functionals on $\Omega$, and we have
	\begin{equation*}
	0\leq I(\varphi,\varphi')\leq (m+1)(I(\varphi,\varphi')-J(\varphi,\varphi')) \leq m I(\varphi,\varphi').
	\end{equation*}
\end{enumerate}
\end{prop}
The propositions \ref{LM} and \ref{IJ} can be proved by a similar method as in the K\"ahler cases (see \cite{M1})
by applying 
the procedure in the proof of the lemma \ref{volume} in the secion 3.
\begin{defn}\label{Hami-holo-vec}
A complex vector field $X$ on a Sasakian manifold is called a Hamiltonian holomorphic vector field if
\begin{enumerate}
\item $d\pi_{\alpha}(X)$ is a holomorphic vector field on $V_{\alpha}$.
\item The basic function $u_{X}:=\sqrt{-1}\eta(X)$ satisfies
\begin{equation*}
\bar{\partial}_{B}u_{X}=-\frac{\sqrt{-1}}{2}i(X)d\eta.
\end{equation*}
Such a function $u_{X}$ is called a Hamiltonian function.
\end{enumerate}
\end{defn}

Let $(x,z^{1},\ldots,z^{m})$ be a foliation chart on $U_{\alpha}$.
Then we can write a Hamiltonian holomorphic vector field $X$ as
\begin{equation*}
X=\eta(X)\frac{\partial}{\partial x}+\sum_{i=1}^{m}X^{i}\frac{\partial}{\partial z^{i}}
-\eta\left(\sum_{i=1}^{m}X^{i}\frac{\partial}{\partial z^{i}}\right)\frac{\partial}{\partial x},
\end{equation*}
where 
\begin{equation*}
\tilde{X}=X+\sqrt{-1}\left(\eta(X)-\eta\left(\sum_{i=1}^{m}X^{i}\frac{\partial}{\partial z^{i}}\right)\right)r\frac{\partial}{\partial r}
\end{equation*}
is a holomorphic vector field on $C(S)$ (see \cite{F1}).

Since $0 \in \Omega$, we abuse a notation as 
\begin{equation*}
M(d\eta_{\varphi})=M(0,\varphi).
\end{equation*}
It is shown that $\varphi$ is a critical point of $M$ on $\Omega$ if and only if $d\eta_{\varphi}$ is a transversely K\"ahler-Einstein metric (see \cite{F1}).

\section{Monge-Amp\`ere equation}

We assume that
 $[\Ric^{T}(d\eta)]_{B}=(2m+2)[d\eta]_{B}$. 
 
Then it follows from the proposition \ref{p-lem} that there exists a function $h$ such that
\begin{align*}
&\Ric^{T}(d\eta)-(2m+2)d\eta=\sqrt{-1}\partial_{B}\bar{\partial}_{B}h\qquad\text{}\\
&\int_{S}(e^{h}-1)(d\eta)^{m}\wedge\eta=0.
\end{align*}
As in K\"ahler geometry,  the Ricci curvature of $d\eta_{u}=d\eta+\sqrt{-1}\partial_{B}\bar{\partial}_{B}u$ is given by
\begin{align*}
\Ric^{T}(d\eta_{u})
=&-\sqrt{-1}\partial_{B}\bar{\partial}_{B}\log(d\eta_{u})^{m}\\
=&-\sqrt{-1}\partial_{B}\bar{\partial}_{B}\log\left(\frac{(d\eta_{u})^{m}}{(d\eta)^{m}}\right)+\Ric^{T}(d\eta)\\
=&-\sqrt{-1}\partial_{B}\bar{\partial}_{B}\log\left(\frac{(d\eta_{u})^{m}}{(d\eta)^{m}}\right)\\
&\qquad	+\sqrt{-1}\partial_{B}\bar{\partial}_{B}(-(2m+2)u+h)+(2m+2)d\eta_{u}.
\end{align*}
Hence, $d\eta_{u}$ is a transversely K\"ahler-Einstein metric if and only if $d\eta_{u}$ satisfies the following  equation
\begin{align*}
-\sqrt{-1}\partial_{B}\bar{\partial}_{B}\log\left(\frac{(d\eta_{u})^{m}}{(d\eta)^{m}}\right)
+\sqrt{-1}\partial_{B}\bar{\partial}_{B}(-(2m+2)u+h)=0\\
\end{align*}
which is equivalent to the Monge-Amp\`ere equation,
\begin{align*}
\frac{(d\eta+\sqrt{-1}\partial_{B}\bar{\partial}_{B}u)^{m}}{(d\eta)^{m}}=\exp(-(2m+2)u+h)
\end{align*}
In order to prove the uniqueness of solutions, we consider two families of equations parametrized by $t \in [0,1]$: 
\begin{align}
\frac{(d\eta+\sqrt{-1}\partial_{B}\bar{\partial}_{B}u)^{m}}{(d\eta)^{m}}&=\exp(-t(2m+2)u+h)\label{s1}\\
\frac{(d\eta+\sqrt{-1}\partial_{B}\bar{\partial}_{B}u)^{m}}{(d\eta)^{m}}&=\exp(-t(2m+2)u-(2m+2)L(0,u)+h)\label{s2}.
\end{align}
For a solution $u$ of (\ref{s1}), $\displaystyle{u-\frac{1}{t+1}L(0,u)}$ is a solution of (\ref{s2}).
On the other hands, if $t>0$, for a solution $u$ of (\ref{s2}), $\displaystyle{u+\frac{1}{t}L(0,u)}$ is a solution of (\ref{s1}).
Therefore (\ref{s1}) and (\ref{s2}) are same for $t \in (0,1]$, but a difference occurs to $t=0$.
If $u$ is a solution of (\ref{s1}), then $u+(\text{constant})$ is also a solution for $t=0$, but not for $t>0$.
Sicne this is inconvenient to prove the uniqueness, we introduce
the equation (\ref{s2}) for this problem.
We set
\begin{align*}
I_{1}=&\{t \in [0,1]\;|\;\text{ the equation }\mbox{(\ref{s1}) has solutions for $t$}\}\\
I_{2}=&\{t \in [0,1]\;|\;\text{ the equation }\mbox{(\ref{s2}) has solutions for $t$}\}.
\end{align*}

If we prove $I_{1}$ is open and close, then there exists a solution for $t=1$ and 
this solution gives a transversely K\"ahler-Einstein metric.

We remark that a solution $u$ of (\ref{s1}) or (\ref{s2}) satisfies
\begin{gather*}
\Ric^{T}(d\eta_{u})=t(2m+2)d\eta_{u}+(1-t)(2m+2)d\eta\\
\therefore \Ric^{T}(d\eta_{u})\geq t(2m+2)d\eta_{u}.
\end{gather*}

From now on we always assume that $\Ric^{T}(d\eta)-(2m+2)d\eta=\sqrt{-1}\partial_{B}\bar{\partial}_{B}h$.

\subsection{Openness}

In this subsection, we shall prove that $I_{2}$ is open.

\begin{defn}
A Hamiltonian holomorphic vector field $X$ is called a normalized Hamiltonian holomorphic vector field
if the Hamiltonian function $u_{X}$ satisfies
\begin{equation*}
\int_{S}u_{X}e^{h}(d\eta)^{m}\wedge\eta=0.
\end{equation*}
\end{defn}

\begin{prop}[theorem 5.1 of \cite{F1}]\label{lap-ei}
Let $\square_{h}^{B}$ be the Laplacian with respect to Hermitian metric $\exp(h)d\eta$.
Then we have
\begin{enumerate}
\item The first eigenvalue of $\square_{h}^{B}$ is greater than or equal to $2m+2$.
\item $\Ker(\square_{h}^{B}-(2m+2))$ is isomorphic to
	$\{X$ $\vert$ normalized Hamiltonian holomorphic vector fields$\}$.
	The correspondence is given by
\begin{equation*}
u\mapsto u\xi+\sum(g^{T})^{i\bar{j}}\frac{\partial u}{\partial \bar{z}^{j}}\frac{\partial}{\partial z^{i}}
	+\eta\left(\sum(g^{T})^{i\bar{j}}\frac{\partial u}{\partial \bar{z}^{j}}\frac{\partial}{\partial z^{i}}\right)\xi.
\end{equation*}
\end{enumerate}
\end{prop}

Let $V$ be a open subset in $\mathbb{C}^{m}$ and $\exp(h)(d\eta)^{m}$ a Hermitian metric on the anti-canonical line bundle $K_{V}^{-1}$.
We denote by $\square_{K_{V}^{-1},h}^{\bar{\partial}}$ the Laplacian with respect to this metirc .
Let $R_{V,h}$ be the curvature of the canonical connection.
Then we have
\begin{equation*}
\sqrt{-1}R_{V,h}=(2m+2)d\eta.
\end{equation*}
By the Kodaira-Akitsuki-Nakano identity on $\Lambda^{m,1}(K_{V}^{-1})$, we have
\begin{equation*}
\square_{K_{V}^{-1},h}^{\bar{\partial}}=\square_{K_{V}^{-1},h}^{\partial}+(2m+2).
\end{equation*}

\begin{prop}\label{I-open}
The intervals $I_1$ and $I_2$ satisfy the followings,
\begin{enumerate}
\item[(i)]$0 \in I_{1}$, $I_{2}$.
\item[(ii)] $I_{2}$ is a open set in $[0,1)$.
\item [(iii)] If $S$ doesn't have non-trivial normalized Hamiltonian holomorphic vector fields, 
both $I_1$ and $I_{2}$ are open in a neighborhood of $1$.
\end{enumerate}
\end{prop}

\begin{proof}

At first we shall show (i).
The equation (8) admits a solution $u$ for $t=0$ by \cite{Y1} and \cite{E1}.
Thus $0\in I_1$.
For a solution $u$ of $I_{1}$ of $t=0$, $u-L(0,u)$ is a solution of $I_{2}$ of $t=0$.

Next we shall show (ii).
We define $\Phi_1$ for  (\ref{s1}) by 
\begin{align*}
&\Phi_{1}\colon\Omega\times I\rightarrow C^{0,\varepsilon}_{B}(S)\\ \\
\Phi_{1}(u,t)=\log&\left(\frac{(d\eta+\sqrt{-1}\partial_{B}\bar{\partial}_{B}u)^{m}}{(d\eta)^{m}}\right)+t(2m+2)u-h,
\end{align*}
where $u \in C^{2,\varepsilon}_{B}(S)$.
We also define $\Phi_2$ for (\ref{s2}) by 
\begin{gather*}
\Phi_{2}\colon\Omega\times I\rightarrow C^{0,\varepsilon}_{B}(S)\\ \\
\begin{split}
\Phi_{2}(u,t)=&\log\left(\frac{(d\eta+\sqrt{-1}\partial_{B}\bar{\partial}_{B}u)^{m}}{(d\eta)^{m}}\right)\\
&\qquad+t(2m+2)u+(2m+2)L(0,u)-h.
\end{split}
\end{gather*}
where $u \in C^{2,\varepsilon}_{B}(S)$.
By differentiating $\Phi_1$ and $\Phi_2$ in the $\dot{u}$ direction at $u$ for a fixed $t$ , we have
\begin{align*}
(d\Phi_{1})_{u}(\dot{u})&=-\square_{u}^{B}\dot{u}+t(2m+2)\dot{u}\\
(d\Phi_{2})_{u}(\dot{u})&=-\square_{u}^{B}\dot{u}+t(2m+2)\dot{u}+\frac{2m+2}{V}\int_{S}\dot{u}(d\eta_{u})^{m}\wedge\eta.
\end{align*}
Our discussion is divided into two cases : $t=0$ and $t\neq 0$.
\begin{enumerate}
\item In the case of  $t=0$.

Let $u$ be a solution of (\ref{s2}) and $\dot{u}\in\Ker (d\Phi_{2})_{u}$.
Then we have
\begin{equation*}
\square_{u}^{B}\dot{u}=\frac{2m+2}{V}\int_{S}\dot{u}(d\eta_{u})^{m}\wedge\eta.
\end{equation*}
When we integrate this in $S$, we have
\begin{align*}
0=&\int_{S}\square_{u}^{B}\dot{u}(d\eta_{u})^{m}\wedge\eta\\
=&\int_{S}\left(\frac{2m+2}{V}\int_{S}\dot{u}(d\eta_{u})^{m}\wedge\eta\right)(d\eta_{u})^{m}\wedge\eta\\
=&(2m+2)\int_{S}\dot{u}(d\eta_{u})^{m}\wedge\eta.
\end{align*}
Therefore we have
\begin{equation*}
\square_{u}^{B}\dot{u}=\frac{2m+2}{V}\int_{S}\dot{u}(d\eta_{u})^{m}\wedge\eta=0.
\end{equation*}
Hence $\dot{u}$ is a constant.
Since the integration is $0$, the constant is $0$.
Therefore $\dot{u}=0$.
From the implicit function theorem, there exists an open neighborhood of $0$ which is included in $I_{2}$.

\item In the case of $t\in (0,1]$.

Since (\ref{s2}) is not different from (\ref{s1}) for $t\in (0,1]$,  It suffices to prove that for $I_{1}$.
Let $u$ be a solution of (\ref{s1}) with $\dot{u}\in\Ker (d\Phi_{1})_{u}$.
Then we have
\begin{equation*}
\square_{u}^{B}\dot{u}=t(2m+2)\dot{u}.
\end{equation*}

We consider $\bar{\partial}_{B}\dot{u}$ to be an element of $\Lambda^{m,1}(K_{V}^{-1})$.
Then by the Kodaira-Akitsuki-Nakano identity, we have
\begin{equation*}
\square_{u}^{B}=\square_{K_{V}^{-1},u}^{\bar{\partial}}=\square_{K_{V}^{-1},u}^{\partial}+[\Ric^{T}(d\eta_{u}),\Lambda].
\end{equation*}
Therefore we have
\begin{align*}
\frac{1}{2m+2}t\Vert \bar{\partial}_{B}\dot{u}\Vert^{2}
=&(\square_{u}^{B}\bar{\partial}_{B}\dot{u},\bar{\partial}_{B}\dot{u})\\
=&(\square_{K_{V}^{-1},u}^{\partial}\bar{\partial}\dot{u},\bar{\partial}\dot{u})
	+([\Ric^{T}(d\eta_{u}),\Lambda]\bar{\partial}\dot{u},\bar{\partial}\dot{u})\\
=&(\square_{K_{V}^{-1},u}^{\partial}\bar{\partial}\dot{u},\bar{\partial}\dot{u})
	+(\Ric^{T}(d\eta_{u})\Lambda\bar{\partial}\dot{u},\bar{\partial}\dot{u}).
\end{align*}
In the case of $0<t<1$, there exists a positive constant $\varepsilon$ such that
\begin{equation*}
\Ric^{T}(d\eta_{u})> (t+\varepsilon)(2m+2)d\eta_{u}.
\end{equation*}
Since $(\square_{K_{V}^{-1},u}^{\partial}\bar{\partial}\dot{u},\bar{\partial}\dot{u})\geq 0$, we have
\begin{align*}
0\leq& t(2m+2)\Vert \bar{\partial}\dot{u}\Vert^{2}-(\Ric^{T}(d\eta_{u})\Lambda\bar{\partial}\dot{u},\bar{\partial}\dot{u})\\
<&t(2m+2)\Vert \bar{\partial}\dot{u}\Vert^{2}-((t+\varepsilon)(2m+2)L\Lambda\bar{\partial}\dot{u},\bar{\partial}\dot{u})\\
=&t(2m+2)\Vert \bar{\partial}\dot{u}\Vert^{2}-(t+\varepsilon)(2m+2)([L,\Lambda]\bar{\partial}\dot{u},\bar{\partial}\dot{u})\\
=&t(2m+2)\Vert \bar{\partial}\dot{u}\Vert^{2}-(t+\varepsilon)(2m+2)(\bar{\partial}\dot{u},\bar{\partial}\dot{u})\\
=&-\varepsilon(2m+2)\Vert \bar{\partial}\dot{u}\Vert^{2}<0
\end{align*}
Hence $\bar{\partial}\dot{u}=0$.
Thus $\dot{u}$ is a constant and we have
\begin{gather*}
0=-\square_{u}^{B}\dot{u}+t(2m+2)\dot{u}=t(2m+2)\dot{u}\\
\therefore \dot{u}=0.
\end{gather*}
By the implicit function theorem, $I_{1}\cap (0,1)$ is a open set.
Therefore $I_{2}\cap (0,1)$ is a open set.
\end{enumerate}

Finally we shall show (iii).
It also suffices to show that for $I_{1}$.
In the case of $t=1$, solutions are transversely K\"ahler-Einstein metrics.
If $S$ doesn't have non-trivial normalized Hamiltonian holomorphic vector fields, We have $\Ker(\square_{u}^{B}-(2m+2))=0$ by proposition \ref{lap-ei}.
Hence, by the implicit function theorem, $I_{1}$ is an open set in neighborhood of $1$.
Therefore $I_{2}$ is also an open set in neighborhood of $1$.
\end{proof}

\subsection{Estimates for closeness}
In this subsection, we shall obtain estimates to prove that $I_{2}$ is close.

By Yau \cite{Y1} and El-Kacimi \cite{E1},
if there exists a $C^0$-estimate of solutions of the Monge-Amp\`ere equation
\begin{equation*}
\sup_{S}\vert u\vert \leq C,
\end{equation*}
then we obtain a $C^{2,\varepsilon}$-estimate,
\begin{equation*}
\Vert u\Vert_{C^{2,\varepsilon}} \leq C'
\end{equation*}
where $C'$ is a constant which dosen't depend on $u$.
Afterward, it follows form the Ascoli-Arzel\`a theorem that $I_{2}$ is closed.

\begin{lem}\label{M-I-J}
Let $u_{t}$ be a $C^{\infty}$-solution of {\rm (\ref{s2})}.
Then we have
\begin{equation*}
\frac{dM(0,u_{t})}{dt}=-(2m+2)(1-t)\frac{d}{dt}(I(0,u_{t})-J(0,u_{t}))\leq 0.
\end{equation*}
\end{lem}

\begin{proof}
Let $\eta_{t}=\eta+d_{B}^{c}u_{t}$.
By the definition of $h$ and (\ref{s2}), we have
\begin{gather*}
\Ric^{T}(d\eta_{t})=(2m+2)d\eta_{t}-\sqrt{-1}(2m+2)(1-t)\partial_{B}\bar{\partial}_{B}u_{t}\\
\end{gather*}
By taking the trace with respect to the transversal K\"ahler form $d\eta_t$, we have 
\begin{equation*}
s^{T}(d\eta_{t})-m(2m+2)=(2m+2)(1-t)\square_{u_{t}}^{B}u_{t}.
\end{equation*}
Hence we obtain
\begin{align*}
\frac{dM(0,u_{t})}{dt}
=&-\frac{1}{V}\int_{S}\dot{u}_{t}(s^{T}(d\eta_{t})-m(2m+2))(d\eta_{t})^{m}\wedge\eta\\
=&-\frac{1}{V}\int_{S}\dot{u}_{t}(2m+2)(1-t)\square_{u_{t}}^{B}u_{t}(d\eta_{t})^{m}\wedge\eta\\
=&-(2m+2)(1-t)\frac{1}{V}\int_{S}u_{t}\square_{u_{t}}^{B}\dot{u}_{t}(d\eta_{t})^{m}\wedge\eta.
\end{align*}
Therefore the first equation of this lemma is proved by proposition \ref{IJ}.

We take a logarithm of (\ref{s2}) and differentiate by $t$.
Then we have
\begin{equation*}
-\square^{B}_{u_{t}}\dot{u}_{t}=-(2m+2)\left(u_{t}+t\dot{u}_{t}+\frac{1}{V}\int_{S}\dot{u_{t}}(d\eta_{t})^{m}\wedge\eta\right).
\end{equation*}
Therefore we obtain
\begin{align*}
&(2m+2)\frac{d}{dt}(I(0,u_{t})-J(0,u_{t}))\\
=&\frac{(2m+2)}{V}\int_{S}\dot{u}_{t}\square_{u_{t}}^{B}u_{t}(d\eta_{t})^{m}\wedge\eta\\
=&\frac{1}{V}\int_{S}\dot{u}_{t}\square_{u_{t}}^{B}
	\left(\square^{B}_{u_{t}}\dot{u}_{t}-t(2m+2)\dot{u}_{t}-\frac{(2m+2)}{V}\int_{S}\dot{u_{t}}(d\eta_{t})^{m}\wedge\eta\right)
	(d\eta_{t})^{m}\wedge\eta\\
=&\frac{1}{V}\int_{S}\dot{u}_{t}\bar{\partial}_{B}^{*}\bar{\partial}_{B}
	\left(\bar{\partial}_{B}^{*}\bar{\partial}_{B}\dot{u}_{t}-t(2m+2)\dot{u}_{t}\right)(d\eta_{t})^{m}\wedge\eta\\
=&\frac{1}{V}\int_{S}
	\left(\bar{\partial}_{B}\dot{u}_{t},\square_{u_{t}}^{B}\bar{\partial}_{B}\dot{u}_{t}-t(2m+2)\bar{\partial}_{B}\dot{u}_{t}\right)
	(d\eta_{t})^{m}\wedge\eta.
\end{align*}
Here we assume $\bar{\partial}_{B}\dot{u}_{t}$ as an element of $A^{m,1}(K_{V}^{-1})$ as in proposition \ref{I-open}.
Then we have
\begin{align*}
&\left(\bar{\partial}_{B}\dot{u}_{t},\square_{u_{t}}^{B}\bar{\partial}_{B}\dot{u}_{t}-t(2m+2)\bar{\partial}_{B}\dot{u}_{t}\right)\\
=&\left(\square_{K_{V}^{-1},u_{t}}^{\partial}\bar{\partial}\dot{u}_{t},\bar{\partial}\dot{u}_{t}\right)
	+\left(([\Ric^{T}(d\eta_{u_{t}}),\Lambda]-(2m+2)t)\bar{\partial}\dot{u}_{t},\bar{\partial}\dot{u}_{t}\right)\\
\geq&\left(\square_{K_{V}^{-1},u_{t}}^{\partial}\bar{\partial}\dot{u}_{t},\bar{\partial}\dot{u}_{t}\right)
	+(2m+2)\left([tL,\Lambda]-t)\bar{\partial}\dot{u}_{t},\bar{\partial}\dot{u}_{t}\right)\\
\geq&0.
\end{align*}
Therefore we obtain
\begin{equation*}
\frac{dM(0,u_{t})}{dt}=-(2m+2)(1-t)\frac{d}{dt}(I(0,u_{t})-J(0,u_{t}))\leq 0.
\end{equation*}
\end{proof}
The following is a well known fact on the Green function on compact Riemannian manifolds

\begin{fac}[\cite{N1}]\label{Green}
Let $(S,g)$ be a compact Riemannian manifold of dimension $2m+1$.
Then there exists a Green function $G(x,y)$ which satisfyies  
\begin{align*}
u(x)=&\frac{1}{\Vol(S,g)}\int_{S}u(y)dV_{g}(y)+\int_{S}G(x,y)(\triangle u)(y)dV_{g}(y)\\
\text{for all } u\in C^{\infty}(S) \text{  and }\\
&\int_S G(x,y) dV_g(y)=0,
\end{align*}
where $\triangle$ is the Lapacian and $dV_g$ is the volume form.
In addition, we assume
\begin{equation*}
\diam(X,g)^{2}\Ric(g)\geq -(m-1)\varepsilon^{2}g
\end{equation*}
for a constant $\varepsilon \geq 0$.
Then there exists a constant $\gamma(m,\varepsilon)$ which depends on only $m$ and $\varepsilon$
and we have
\begin{equation*}
G(x,y)\geq -\gamma(m,\varepsilon)\frac{\diam(S,g)^{2}}{\Vol(S,g)}
\end{equation*}
for the Green function of $(S,g)$.
\end{fac}

We remark that a solution $u$ of (\ref{s1}) or (\ref{s2}) satisfying
\begin{gather*}
\Ric^{T}(d\eta_{u})=t(2m+2)d\eta_{u}+(1-t)(2m+2)d\eta\\
\end{gather*}
Hence we have 
\begin{equation*}
 \Ric^{T}(d\eta_{u})\geq t(2m+2)d\eta_{u}.
\end{equation*}
We introduce a family of contact structures by the multiplication of positive constant $\mu$,
\begin{align}\label{mu-fmily of Sasakina}
\eta_{u,\mu}&=\mu^{-1}\eta_{u}\\
\xi_{\mu}&=\mu\xi
\end{align}
Then we see that $(\eta_{u,\mu},\xi_{\mu})$ gives a Sasakian structure with the metric 
$g_{u,\mu}$ on $S$.
The transversal metric $g^{T}_{u,\mu}$ is given by $g^{T}_{u,\mu}=\mu^{-1}g_{u}^{T}$.
The volume form of $g_{u,\mu}$ is given by
\begin{equation}\label{eq: volume}
\eta_{u,\mu}\wedge(d\eta_{u,\mu})^{m}=\mu^{-(m+1)}\eta_{u}\wedge(d\eta_{u})^{m}.
\end{equation}
Let  $\square_{u,\mu}$ be 
the Laplacian with the Green operator $G_{u,\mu}$ 
and $\Ric_{u,\mu}$ the Ricci tensor with respect to $g_{u,\mu}$.

\begin{prop}\label{v-d}
Let $(S,g)$ be a compact Sasakian manifold and $u$ a solution of (\ref{s1}) or (\ref{s2}).
If we set  $\mu=t^{-1}$, then we have estimates of the volume and the diameter with respect to the metric $g_{u,\mu}$,
\begin{align*}
&\Vol(S, g_{u,\mu})=t^{m+1}V\\
&\diam(S,g_{u,\mu}) \leq \pi
\end{align*}
where $V$ is the volume of $S$ with respect to $(d\eta)^{m}\wedge \eta$.
\end{prop}

\begin{proof}
Since $\mu=t^{-1}$, lemma \ref{volume} yields, 
\begin{align*}
\Vol(S, g_{u,\mu})
&=\int_{S}(d\eta_{u,\mu})^{m}\wedge \eta_{u,\mu}=\mu^{-(m+1)}\int_{S}(d\eta_{u})^{m}\wedge \eta_{u}\\
&=t^{m+1}\int_{S}(d\eta)^{m}\wedge \eta\\
&=t^{m+1}V.
\end{align*}

By theorem \ref{R-R^{T}}, we have
\begin{equation*}
\Ric_{u,\mu}(X,Y)=\Ric^{T}_{u,\mu}(X,Y)-2g_{u,\mu}(X,Y)\qquad \forall X,Y \in \Ker\eta_{u,\mu},
\end{equation*}
and by definition of $\eta_{u,\mu}$,
\begin{equation*}
\mu g_{u,\mu}(X,Y)=g_{u}(X,Y)\qquad \forall X,Y \in \Ker\eta_{u,\mu}.
\end{equation*}
Since the transversal Ricci curvature is invariant under the multiplication by positive constant of a transversal metric,
thus $\Ric^{T}_{u,\mu}=\Ric^{T}_{u}$, for all $X,Y \in \Ker\eta_{u,\mu}$.
Then we  have
\begin{align*}
\Ric^{T}_{u,\mu}(X,Y) &=\Ric^{T}_{u}(X,Y)\\
&\geq t(2m+2)g^{T}_{u}(X,Y)\\
&=(2m+2)t\mu g^{T}_{u,\mu}(X,Y)\\
&=(2m+2)t\mu g_{u,\mu}(X,Y).
\end{align*}
Therefore we have
\begin{equation*}
\Ric_{u,\mu}(X,Y)\geq (2m+2)t\mu g_{u,\mu}(X,Y)-2g_{u,\mu}(X,Y)\qquad \forall X,Y \in \Ker\eta_{u,\mu}.
\end{equation*}
Since we set
\begin{equation*}
\mu=t^{-1}.
\end{equation*}
we have
\begin{equation*}
\Ric_{u,\mu}(X,Y) \geq 2m g_{u,\mu}(X,Y)\qquad \forall X,Y \in \Ker\eta_{u,\mu}.
\end{equation*}
It follows from theorem 2.4 that 
\begin{align*}
\Ric_{u,\mu}(X,\xi_\mu)=&2m\, \eta_{u,\mu}(X)\\
=&2m\, g_{u,\mu}(X,\xi_\mu),\qquad \forall X\in TS
\end{align*}
Therefore we obtain
\begin{equation*}
\Ric_{u,\mu} \geq 2m g_{u,\mu} \geq (m-1)g_{u,\mu}.
\end{equation*}
By the Myers theorem, we have
\begin{equation*}
\diam(S,g_{u,\mu}) \leq \pi.
\end{equation*}
\end{proof}

\begin{lem}\label{basic Lap and Lap}
Let $\triangle$ be a compact Sasakian manifold $(S, g, \eta, \xi)$ and $\triangle_g$ the Laplacian with respect to
$g$ on $S$. The transversal K\"ahler metric $d\eta$ on $S$ gives the basic Laplacian $\triangle^B_{d\eta}$
and the basic complex Laplacian $\square^B_{d\eta}$ on $S$.
Then we have 
\begin{equation*}
\triangle_g u= \triangle^B_{d\eta} u=2\square^B_{d\eta}u, 
\end{equation*}
for every basic function $u$ on $S$.
\end{lem}
\begin{proof}
As in K\"ahler geometry, we have $\triangle^B_{d\eta} u=2\square^B_{d\eta}u$. 
There is a relation between the Hodge star operator $*$ and the basic Hodge star operator $*_B$, 
\begin{align}
&*(\eta\wedge \alpha)=  *_{B} \alpha,\\
&*\alpha=(-1)^p\eta\wedge *_{B} \alpha
\end{align}
for a basic $p$-form $\alpha$.
Since $d\eta$ is a basic $2$-form, we have $d\eta\wedge *_{B}du =0$.
Then we obtain
\begin{align}
\triangle_{g} u =&-*d *du =-*d (-\eta\wedge *_{B}du)\\
=&-*(\eta\wedge d *_{B} du )\\
=&=-*_{B} d_{B} *_{B} d_{B}u = \triangle^{B} u
\end{align}
\end{proof}

We can estimate $u$ with the functional $I$.

\begin{lem}\label{osc-a}
Let $(S,g, \eta, \xi)$ be a compact Sasakian manifold and $u_{t}\;(0<t\leq 1)$ a family of  basic functions
with transversal K\"ahler form $d\eta_{u_t}$.
We assume $\Ric^{T}(d\eta_{u_{t}})\geq (2m+2)td\eta_{u_{t}}$.
Let $G_g$ be the Green function with respect to $g$ with a lower bound $\inf G_g\geq -K$.
Then there exists a constant $C$ which doesn't depend on $t$ such that
\begin{align*}
\osc_{S}u_{t}
=&\sup_{S}u_{t}-\inf_{S}u_{t}\\
\leq&I(0,u_{t})+2m\left(\frac{KV}{m!}+\frac{C}{t}\right)
\end{align*}
where $V$ is the volume of $S$ with respect to $(d\eta)^{m}\wedge\eta$.
\end{lem}

\begin{proof}
Since $d\eta_{u_{t}}=d\eta+\sqrt{-1}\partial\bar{\partial}u_{t}$ is 
a transversal K\"ahler form, we have
\begin{equation*}
\square^{B}_{d\eta}u_{t}=\tr_{d\eta}(d\eta-d\eta_{u_{t}})\leq m.
\end{equation*}
Then applying the fact 5.5 and lemma \ref{basic Lap and Lap}, we have an upper bound of $u_t$
\begin{align}
u_{t}(x)
=&\frac{1}{V}\int_{S}u_{t}(d\eta)^{m}\wedge\eta+\int_{S}\left(G_g(x,y)+K\right)(\triangle_g u_{t})\frac{(d\eta)^{m}\wedge\eta}{m!}\\
=&\frac{1}{V}\int_{S}u_{t}(d\eta)^{m}\wedge\eta+\int_{S}\left(G_g(x,y)+K\right)(2\square^{B}_{d\eta}u_{t})\frac{(d\eta)^{m}\wedge\eta}{m!}\\
\leq&\frac{1}{V}\int_{S}u_{t}(d\eta)^{m}\wedge\eta+2mK\frac{V}{m!}.\label{eq: upper bound}
\end{align}
Let $\triangle_{t,\mu}$ be the Laplacian and
$G_{t,\mu}$ the Green function with respect to the Sasakian metric $g_{u_{t}, \mu}$ defined by 
(\ref{mu-fmily of Sasakina}).
By proposition \ref{v-d} and fact \ref{Green}, we have
\begin{equation}\label{estimae of Green}
G_{{{t},\mu}}\geq -\gamma(m,\varepsilon)\frac{\diam(S,g_{u_{t},\mu})^{2}}{\Vol(S,g_{u_{t},\mu})}
\geq -\gamma(m,0)\frac{\pi^{2}}{t^{m+1}V},
\end{equation}
where $\mu=t^{-1}$ as in proposition \ref{v-d}.
We denote by $2\square^B_{t,\mu}$ the basic complex Laplacian with respect to the transversal K\"ahler form 
$d\eta_{u_t, \mu}$.
Then it follows from lemma \ref{basic Lap and Lap} that 
$\triangle_{t,\mu} u_t =2\square^B_{t,\mu} u_t$.   
By $d\eta_{u_{t},\mu}=\mu^{-1}\left(d\eta+\sqrt{-1}\partial\bar{\partial}u_{t}\right)$, we have
\begin{equation}\label{estimate of 18}
\square^{B}_{t,\mu}u_{t}
=\mu\square^{B}_{t,\mu}\mu^{-1}u_{t}
=\mu\tr_{d\eta_{u_{t},\mu}}(d\eta_{\mu}-d\eta_{u_{t},\mu})\geq -mt^{-1}
\end{equation}
where $\eta_\mu=\mu^{-1}\eta$ and $\mu=t^{-1}$.
By applying the fact 5.5 to $(S, g_{u_t, \mu})$, we have 
\begin{align}
u_{t}(x)
=&\frac{1}{t^{m+1}V}\int_{S}u_{t}(d\eta_{u_{t},\mu})^{m}\wedge\eta_{u_{t},\mu}\label{eq: 20}\\
&+\int_{S}\left(G_{t,\mu}(x,y)+\gamma(m,0)\frac{\pi^{2}}{t^{m+1}V}\right)
	(\triangle_{t,\mu}u_{t})\frac{(d\eta_{u_{t},\mu})^{m}\wedge\eta_{u_{t},\mu}}{m!}\label{eq:  21}
\end{align}
By (\ref{eq: volume}), the first term (\ref{eq: 20}) is given by 
\begin{equation*}
\frac{1}{t^{m+1}V}\int_{S}u_{t}(d\eta_{u_{t},\mu})^{m}\wedge\eta_{u_{t},\mu}
=\frac{1}{V}\int_{S}u_{t}(d\eta_{u_{t}})^{m}\wedge\eta
\end{equation*}
By using  (\ref{estimae of Green}) and (\ref{estimate of 18}), we have an estimate of (\ref{eq: 21}),
\begin{align*}
&\int_{S}\left(G_{t,\mu}(x,y)+\gamma(m,0)\frac{\pi^{2}}{t^{m+1}V}\right)
	(\triangle_{t,\mu}u_{t})\frac{(d\eta_{u_{t},\mu})^{m}\wedge\eta_{u_{t},\mu}}{m!}\\
=&\int_{S}\left(G_{t,\mu}(x,y)+\gamma(m,0)\frac{\pi^{2}}{t^{m+1}V}\right)
	(2\square^{B}_{t,\mu}u_{t})\frac{(d\eta_{u_{t},\mu})^{m}\wedge\eta_{u_{t},\mu}}{m!}\\ \\
\geq&-\frac{2m}{t}\gamma(m,\varepsilon)\frac{\pi^{2}}{t^{m+1}V}\frac{t^{m+1}V}{m!}\\
=&-2m\gamma(m,0)\frac{\pi^{2}}{t(m!)}\\
\end{align*}
Thus we obtain
\begin{equation}\label{eq: lower bound}
u_{t}(x)\geq \frac{1}{V}\int_{S}u_{t}(d\eta_{u_{t}})^{m}\wedge\eta-2m\gamma(m,0)\frac{\pi^{2}}{t(m!)}
\end{equation}
From the upper bound (\ref{eq: upper bound}) and the lower bound (\ref{eq: lower bound}) 
we have the desired estimate
\begin{align*}
\osc_{S}u_{t}
=&\sup_{S}u_{t}-\inf_{S}u_{t}\\
\leq&I(0,u_{t})+2m\left(\frac{KV}{m!}+\frac{C}{t}\right).
\end{align*}
\end{proof}

\begin{lem}\label{I-c2}
Let $(S,g)$ be a compact Sasakian manifold.
If we have a constant $A$ such that
\begin{equation*}
I(0,u_{t})\leq A
\end{equation*}
for a solution $u_{t}$ of {\rm (\ref{s2})},
then we have a constant $B$ which doesn't depend on $t$ such that
\begin{equation*}
\Vert u_{t}\Vert_{C^{2,\varepsilon}}\leq B.
\end{equation*}
\end{lem}

\begin{proof}
We denote by $C_{i}$ a constant which does not depend on $t$.

By lemma \ref{osc-a}, we have
\begin{equation*}
t\osc_{S}u_{t}\leq t\left(I(0,u_{t})+4m\left(\frac{KV}{m!}+\frac{C}{t}\right)\right)
\leq C_{1}.
\end{equation*}
When we integrate (\ref{s2}) on $S$, we have
\begin{align*}
&\int_{S}\exp\left(-t(2m+2)u_{t}-(2m+2)L(0,u_{t})+h\right)(d\eta)^{m}\wedge\eta\\
=&\int_{S}\frac{(d\eta+\sqrt{-1}\partial_{B}\bar{\partial}_{B}u_{t})^{m}}{(d\eta)^{m}}(d\eta)^{m}\wedge\eta\\
=&\int_{S}(d\eta+\sqrt{-1}\partial_{B}\bar{\partial}_{B}u_{t})^{m}\wedge\eta\\
=&\int_{S}(d\eta)^{m}\wedge\eta.
\end{align*}
Therefore there exists a point $x_{t}$ such that
\begin{equation}
-t(2m+2)u_{t}(x_{t})-(2m+2)L(0,u_{t})+h(x_{t})=0.\label{s3}
\end{equation}
Then for all $x \in S$, we have
\begin{align*}
&\vert -t(2m+2)u_{t}(x)-(2m+2)L(0,u_{t})+h(x)\vert\\
=&\vert t(2m+2)u_{t}(x_{t})-t(2m+2)u_{t}(x)-h(x_{t})+h(x)\vert\\
\leq& t\osc_{S}u_{t}+2\sup_{S}\vert h\vert\leq C_{2}.
\end{align*}
Hence we have
\begin{align*}
&\sup_{S}\left\vert \log\frac{(d\eta+\sqrt{-1}\partial_{B}\bar{\partial}_{B}u)^{m}}{(d\eta)^{m}}\right\vert\\
=&\sup_{S}\vert -t(2m+2)u_{t}-(2m+2)L(0,u_{t})+h \vert\\
\leq& C_{2}.
\end{align*}
Thus, by Yau \cite{Y1}, we have
\begin{equation*}
\osc_{S}u_{t}\leq C_{3}.
\end{equation*}
as in compact K\"ahler manifolds.

Next, we shall estimate $\sup_{S}\vert u_{t}\vert$.
We use a path $su_{t}$ to compute the functional $L$ for $0\leq s \leq 1$.
Then we have
\begin{align*}
&\vert L(0,u_{t})-u_{t}(x_{t}) \vert\\
=&\left\vert \frac{1}{V}\int_{0}^{1}\left(\int_{S}(u_{t}-u_{t}(x_{t}))(d\eta+s\sqrt{-1}\partial_{B}\bar{\partial}_{B}u_{t})^{m}\wedge\eta
	\right)ds \right\vert\\
\leq&\frac{1}{V}\int_{0}^{1}\left(\int_{S}\osc_{S}u_{t}(d\eta+s\sqrt{-1}\partial_{B}\bar{\partial}_{B}u_{t})^{m}\wedge\eta\right)ds\\
=&\osc_{S}u_{t}\\
\leq&C_{3}.
\end{align*}
By (\ref{s3}), we have
\begin{align*}
(2m+2)(1+t)\vert u_{t}(x_{t})\vert
=&\vert (2m+2)u_{t}(x_{t})-(2m+2)L(0,u_{t})+h(x_{t})\vert\\
=&(2m+2)\vert u_{t}(x_{t})-L(0,u_{t})\vert +\vert h(x_{t})\vert\\
\leq&C_{4}.
\end{align*}
Thus we have $\vert u_{t}(x_{t})\vert \leq C_{5}$.
Hence for all $x\in S$, we have
\begin{align*}
\vert u_{t}(x)\vert
=&\vert u_{t}(x)-u_{t}(x_{t})+u_{t}(x_{t})\vert\\
\leq&\osc_{S}u_{t}+\vert u_{t}(x_{t})\vert\\
\leq&C_{6}.
\end{align*}
Therefore we obtain
\begin{equation*}
\Vert u_{t}\Vert_{C^{2,\varepsilon}}\leq C_{7}.
\end{equation*}
\end{proof}

\section{Proof of main theorem}

First we prove the next proposition.

\begin{thm}\label{I-uni}
Let $(S,g)$ be a compact Sasakian manifold and $\tau \in (0,1)$.
\begin{enumerate}
\item[\text{{\rm(i)}}] For $t \in [0,\tau]$, if there exists smooth one-parameter families $u_{t}, u'_{t}$ of solutions of (\ref{s2})
	 , then $u_{t}=u'_{t}$.
\item[\text{{\rm(ii)}}] If there exists a solution $u_{\tau}$ of (\ref{s2}) at $t=\tau$,
	 $u_{\tau}$ uniquely extends to a smooth family $\{u_{t}\;|\;0\leq t \leq \tau\}$ of solutions of (\ref{s2}).
\item[\text{{\rm(iii)}}] If there exist two solutions $u_{\tau}, u'_{\tau}$ of (\ref{s2}) at $t=\tau$, then $u_{\tau}=u'_{\tau}$.
\end{enumerate}
\end{thm}

\begin{proof}
\begin{enumerate}
\item[(i)]
By proof of proposition \ref{I-open}, the solution of (\ref{s2}) is locally unique.
Therefore we shall show that $u_{0}$ is unique.
Since $u_{0}$ is a solution of (\ref{s2}) we have
\begin{equation*}
\int_{S}(d\eta)^{m}\wedge\eta=\int_{S}(d\eta_{u_{0}})^{m}\wedge\eta
=\exp(-(2m+2)L(0,u_{0}))\int_{S}e^{h}(d\eta)^{m}\wedge\eta.
\end{equation*}
By definition of $h$, we have
\begin{equation*}
\int_{S}(e^{h}-1)(d\eta)^{m}\wedge\eta=0
\end{equation*}
Therefore we obtain
\begin{equation*}
L(0,u_{0})=0.
\end{equation*}
Then there exist a solution which is unique up to an additive constant by \cite{Y1} and \cite{E1}.
The condition $L(0,u)=0$ says that there is no difference by the additive constant.
Therefore $u_{0}$ is unique and $u_{t}$ is unique also.

\item[(ii)]
Let $I_{2}=\{t \in [0,1]\;|\;\text{the equation (\ref{s2}) has solutions for $t$}\}$.
Since we already showed that $I_{2}$ is open, it suffices to show that $I_{2}$ is closed.
By proposition \ref{IJ} and lemma \ref{M-I-J}, we have
\begin{align*}
I(0,u_{t})
\leq&(m+1)(I(0,u_{t})-J(0,u_{t}))\\
\leq&(m+1)(I(0,u_{\tau})-J(0,u_{\tau})).
\end{align*}
The right hand side is independent of $t$.
Therefore, by lemm \ref{I-c2}, we have
\begin{equation*}
\Vert u_{t}\Vert_{C^{2,\varepsilon}}\leq C.
\end{equation*}
Hence it follows from the Ascoli-Arzel\`a theorem that $I_{2}$ is closed.

\item[(iii)]
From 2, $u_{\tau}, u'_{\tau}$ are extended in $u_{t},u'_{t}\;\;(t\in [0,\tau])$.
From 1 we have $u_{t}=u'_{t}$.
\end{enumerate}
\end{proof}

We are in a position to prove main theorem.
\setcounter{section}{1}
\setcounter{defn}{0}
\begin{thm}
Let $(S,\xi,\eta,\Phi)$ be a compact Sasakian manifold. 
We assume that $S$ doesn't admit nontrivial Hamiltonian holomorphic vector fields.
If $S$ has a Sasaki-Einstein metric, then the Sasaki-Einstein metric is unique. 
In other words,
if there are two Sasaki-Einstein metrics $\omega_{1}$ and $\omega_{2}$ on $S$, then $\omega_{1}=\omega_{2}$.
\end{thm}

\begin{proof}
If we have a solution of (\ref{s1}),
then we have a solution of (\ref{s2}).
Therefore theorem \ref{I-uni} is true for (\ref{s1}).
In particular, by proposition \ref{I-open},
if $S$ doesn't have nontrivial Hamiltonian holomorphic vector fields, then $I_{1}$ is open in $t=1$.
Therefore if there are two Sasaki-Einstein metrics $\omega_{1}$ and $\omega_{2}$,
then $\omega_{1}=\omega_{2}$.
\end{proof}
\setcounter{section}{6}
\setcounter{defn}{1}


\begin{thebibliography}{99}
\bibitem{BM1} S. Bando, T. Mabuchi, Uniqueness of Einstein-K\"ahler metrics modulo connected group actions. Algebraic Geometry, Adv.Studies in Pure math. 10 (1987)
\bibitem{BGM1} C. P. Boyer, K. Galicki and P. Matzeu, On eta-Einstein Sasakian geometry, Comm. Math. Phys., 262 (2006), 177-208
\bibitem{BG1} C. P. Boyer, K. Galicki, Sasaki geometry, Oxford Mathematical Monographs (2008)
\bibitem{CFO1} K. Cho, A. Futaki, H. Ono,Uniqueness and examples of toric Sasaki-Einstein manifolds, Comm. Math. Phys., 277 (2008), 439-458. (math.DG/0701122)
\bibitem{E1} A. El Kacimi-Alaoui, Op\'erateurs transversalement elliptiques sur un feuilletage reimannien et applications, Compositio Math. 73 (1990), 57-106
\bibitem{F1} A. Futaki, H. Ono and G. Wang, Transverse K\"ahler geometry of Sasaki manifolds and toric Sasaki-Einstein manifolds, to appear in J. Differential Geometry. math.DG/0607586
\bibitem{GMSW1} J. P. Gauntlett, D. Martelli, J. Sparks and W. Waldram, Sasaki-Einstein metrics on $S^{2}\times S^{3}$, adv. Theor. Math. Phys., 8 (2004), 711-734
\bibitem{GMSW2} J. P. Gauntlett, D. Martelli, J. Sparks and W. Waldram, A new infinite class of Sasaki-Einstein manifolds, adv. Theor. Math. Phys., 8 (2004), 987-1000
\bibitem{M1} T. Mabuchi, $K$-energy maps integrating Futaki invariants, T\^ohoku Math. Journ. 38 (1986), 575-593
\bibitem{N1} H. Nakajima, Nonlinear PDE and complex geometry (in Japanese), Iwanami Shoten, (2008)
\bibitem{Ni} Y. Nitta, A diameter bound for Sasaki manifolds with application to uniqueness for Sasaki-Einstein structure, math.DG/09060170
\bibitem{Y1} S. T. Yau, On the Ricci curvature of a compact K\"ahler manifold and the complex Monge-Amp\`ere equation I, Comm. Pure Appl. Math., 31 (1978), 339-441
\end{thebibliography}
\end{document}